\documentclass{amsproc}

\newtheorem{theorem}{Theorem}[section]

\newtheorem{prop}[theorem]{Proposition}
\newtheorem{cor}[theorem]{Corollary}

\theoremstyle{definition}
\newtheorem{definition}[theorem]{Definition}
\newtheorem{example}[theorem]{Example}

\theoremstyle{remark}

\numberwithin{equation}{section}

\DeclareMathOperator{\Min}{Min}
\DeclareMathOperator{\Trace}{tr}

\DeclareMathOperator{\diag}{diag}
\DeclareMathOperator{\End}{End}
\DeclareMathOperator{\GL}{GL}
\DeclareMathOperator{\Aut}{Aut}

\DeclareMathOperator{\SL}{SL}

\newcommand{\Z}{{\mathbb{Z}}}
\newcommand{\Q}{{\mathbb{Q}}}

\newcommand{\N}{{\mathbb{N}}}
\newcommand{\R}{{\mathbb{R}}}

\newcommand{\HH}{{\mathbb{H}}}
\newcommand{\C}{{\mathbb{C}}}

\begin{document}

\title{Golden lattices.}

\author{Gabriele Nebe}
\address{
Lehrstuhl D f\"ur Mathematik, RWTH Aachen University,
52056 Aachen, Germany}
\email{nebe@math.rwth-aachen.de}

\dedicatory{This paper is dedicated to Boris Venkov.}


\maketitle

\begin{abstract}
Let $\vartheta := \frac{-1+\sqrt{5}}{2}$ be the golden ratio.
A golden lattice is an even unimodular $\Z[\vartheta ]$-lattice 
of which the Hilbert theta series is an extremal Hilbert modular form.
We construct golden lattices from extremal even unimodular lattices
and obtain families of dense modular lattices. 
\end{abstract}

\section{Introduction} 

\subsection{Even unimodular $R$-lattices}
Let $R$ be the ring of integers in a real number field $K$. 
Let $\Lambda $ be a full $R$-lattice in Euclidean $n$-space
$(K^n,Q)$, so $\Lambda $ is a finitely generated $R$-submodule of 
$K^n$ that spans $K^n$ over $K$ and $Q:K^n \to K$ is a 
totally positive definite quadratic form.
The {\em polar form} of $Q$ is 
the positive definite symmetric $K$-bilinear form $B$ defined by
$B(x,y):=Q(x+y)-Q(x)-Q(y)$. 
The lattice $(\Lambda , Q)$ is called an {\em even unimodular} $R$-lattice,
if $Q:\Lambda \to R$ is an integral quadratic form such that 
$$\Lambda = \Lambda ^{\#} := \{ x\in K^n \mid B(x,\lambda )\in R \mbox{ for all } \lambda \in \Lambda \} .$$
For small dimension and small fields such lattices have been
classified in \cite{Hsiaw5}, \cite{Hsiaw2}, \cite{Hungw3}. 

\subsection{Trace lattices} 
Any $R$-lattice $(\Lambda, Q)$ and any totally positive $\alpha \in K_+$
gives rise to a positive definite 
$\Z $-lattice 
$$L_{\alpha }:= (\Lambda , \Trace_{K/\Q }(\alpha Q)) $$
 of dimension $n[K:\Q ]$.
$L_{\alpha }$ will be called a {\em trace lattice} of $(\Lambda , Q)$, since
the quadratic form $$q:=q_{\alpha } :L_{\alpha } \to \Q, x\mapsto \Trace_{K/\Q }(\alpha Q(x))$$
is obtained as a trace. 
In this way  an $R$-lattice defines a $[K:\Q ]$-parametric family
of positive definite $\Z $-lattices $\{ L_{\alpha } \mid \alpha \in K_+ \}$
of which the most important 
invariants like minimum and determinant, and also the theta series 
may be read off from the corresponding invariants of 
 the lattice $(\Lambda , Q)$ and its Hilbert theta
series. 
For instance the dual lattice
\begin{equation} \label{Zdual} 
L_{\alpha }^* := \{  x\in K^n \mid \Trace _{K/\Q} (\alpha B(x,\lambda )) \in \Z  \mbox{ for all } \lambda \in \Lambda \} = \alpha ^{-1} R^{*} \Lambda ^{\#} 
\end{equation}
where 
$R^{*} := \{ s\in K \mid \Trace_{K/ \Q} (sR ) \subseteq \Z \} $
is the {\em inverse different} of $R$.

\subsection{Extremal even unimodular lattices}
In particular if $(\Lambda , Q)$ is an even unimodular $R$-lattice 
and $\alpha $ is a totally positive generator of  $R^{*} = \alpha R$
then the trace lattice $L_{\alpha }$ is an even unimodular 
integral lattice of dimension $N=n[K:\Q ]$.  
It is well known that the minimum 
$$\min (L,q) = \min \{  q(x) \mid 0\neq x \in L \} $$
of an even unimodular $\Z $-lattice $L$ of dimension $N$ is bounded 
by $\min (L,q) \leq 1 +  \lfloor \frac{N}{24} \rfloor $. 
Lattices that achieve equality are called {\em extremal}. 
If the dimension $N$ is a multiple of 
24 then the extremal even unimodular lattices are densest known lattices.
There are 4 such lattices known, the Leech lattice $\Lambda _{24}$
in dimension 24,
three lattices $P_{48p}$, $P_{48q}$, and $P_{48n}$ of dimension 48 and 
one lattice $\Gamma _{72}$ of dimension 72.

\subsection{Golden lattices}
This article considers the situation where $R=\Z[\vartheta ] $
and $\vartheta = \frac{-1+\sqrt{5}}{2} $ is the golden ratio. 
Then $K$ is a real quadratic number field of minimal possible discriminant 5,
$R$ is a principal ideal domain and  
$$R^{*}  = \eta ^{-1} R \mbox{ where } \eta = 3+\vartheta = \frac{5+\sqrt{5}}{2} $$ is a totally 
positive generator of the  prime ideal of norm 5.
The extremal even unimodular 
lattices $\Lambda _{24}$, $P_{48n} $ and $\Gamma _{72}$ may be obtained
as trace lattices $L_{\eta^{-1} }$ 
 of even unimodular $\Z[\vartheta ]$-lattices. This structure allows to construct 
interesting families of dense modular lattices (Theorem \ref{main}).

\section{Hilbert theta series of golden lattices} \label{s2} 

\subsection{Symmetric Hilbert modular forms}

Let $R:=\Z[\vartheta ]$ be the ring of integers in the real quadratic number
field $K:=\Q [\sqrt{5}]$. 
Then the Hilbert theta series (Definition \ref{HTS})  of an $n$-dimensional
even unimodular $R$-lattice $(\Lambda , Q)$ is a Hilbert modular form
 of weight $n/2$ (see \cite[Section 5.7]{Ebeling}). If $(\Lambda , Q)$ is Galois invariant, then so is its
theta series and hence this
Hilbert modular form is symmetric. 
Hilbert modular forms for $R$ are holomorphic functions on the direct product 
$\HH _K:= \HH \times \HH$ of 2 copies of the upper half plane $$\HH := \{ z\in \C \mid \Im (z) > 0 \} .$$
If $\sigma _1, \sigma _2 $ denote the two embeddings of $K$ into $\R \subset \C $ then 
$\SL_2(R)$ acts on  $\HH _K$ by 
$$(z_1,z_2) \left( \begin{array}{cc} a & b \\ c & d \end{array} \right) := 
\left( \frac{\sigma _1 (a)z_1+\sigma_1(b)}{\sigma _1(c) z_1 + \sigma _1(d)} , 
\frac{\sigma _2 (a)z_2+\sigma_2(b)}{\sigma _2(c) z_2 + \sigma _2(d)} \right) $$ 
and the Galois automorphism just interchanges the two copies of $\HH $.
The ring of symmetric Hilbert modular forms for the group $\SL_2(R)$ is a polynomial ring
$${\mathcal H} := \C [A_2,B_6,C_{10} ] $$ 
where the explicit generators of weight 2, 6, and 10 have been obtained
in \cite{Gundlach} and can be found in \cite{Ebeling}. 
We denote by 
$${\mathcal H}_w:=\{ f\in {\mathcal H} \mid \mbox{ weight of } f = w \} $$
the space of symmetric Hilbert modular forms of weight $w$. 

\begin{definition}\label{HTS}
Let $(\Lambda ,Q )$ be an even 
$R$-lattice. 
Then 
the {\em Hilbert theta series} of $\Lambda $ is
$$\Theta (\Lambda , Q) := \sum _{\lambda \in \Lambda } \exp (2\pi i \Trace _{K/\Q} (z Q(\lambda ) )) = 
1+\sum _{X\in R_{+}} A_X \exp (2\pi i \Trace _{K/\Q} (z X )) $$
where $A_X := |\{ \lambda \in \Lambda \mid Q(\lambda ) = X \} |$ 
 and 
$\Trace _{K/\Q }( z X ) = z_1 \sigma _1(X) + z_2 \sigma _2(X) $ for $z=(z_1,z_2)\in \HH _K$.
\end{definition}

Since $(1,\eta^{-1} )$ is a $\Q $-basis of $K$ the trace of 
$Q(\lambda )$ and $\eta^{-1} Q(\lambda )$ uniquely
 determine the value $Q(\lambda )\in K$.
So $\Theta (\Lambda , Q)$ is determined by the {\em $(q_0,q_1)$-expansion}
$$\Theta (\Lambda , Q) := \sum _{\lambda \in \Lambda } 
q_0^{\Trace_{K/\Q} (\eta^{-1} Q(\lambda ))} q_1 ^{\Trace _{K/\Q} ( Q(\lambda )) } \in
\C [[q_0,q_1 ]]  $$ 
which is very convenient for computations.
Replacing $q_1$ by $1$ yields the usual theta series of the
trace lattice $L_{\eta^{-1} }$ and substituting $q_0$ by 1 gives the
theta series of $L_1$. 
We obtain $A_2(q_0,1) = \Theta (E_8) $ the Eisenstein series of weight 4,
$B_6(q_0,1) = \Delta $, the cusp form of weight 12, and 
$C_{10} (q_0,1 ) = 0 $. 
In particular replacing $q_1$ by 1 yields a surjective 
ring homomorphism onto the ring of elliptic modular forms
of weight divisible by 4 for the full modular group  $\SL_2(\Z )$.

\subsection{Extremal Hilbert modular forms}

\begin{definition}
Define a valuation on the field of fractions of 
$\C [[q_0,q_1]] $ by 
$$ \nu : \C ((q_0 , q_1 )) \setminus \{ 0 \}  \to \Z\times \Z  , 
\nu (\sum _{(i,j) = (s,t)} ^{(\infty,\infty )} A_{(i,j)} q_0^iq_1^j ) := 
\min \{ (i,j) \mid A_{(i,j)} \neq 0 \} $$
where the total ordering on $\Z\times \Z$ is lexicographic, so 
$$(s,t) \leq (s',t') \mbox{ if } s < s' \mbox{ or } s=s' \mbox{ and } t\leq t' .$$
This gives rise to a valuation on the ring of Hilbert modular forms 
via the $(q_0,q_1)$-expansion.
A symmetric Hilbert modular form  $f\in {\mathcal H}_w$
is called an {\em extremal Hilbert modular form of weight $w$},
if 
$$\nu (f-1) \geq \nu (f'-1) \mbox{ for all } f'\in {\mathcal H}_w .$$
\end{definition}

One computes 
$$\nu (A_2-1) = (1,2) ,\ \nu (B_6) = (1,2) , \ \nu (C_{10} ) = (2,4) ,\ \nu (X_{12}) = (2,5) $$
where $X_{12} = \frac{1}{4} (A_2C_{10} - B_6^2) $. 
The valuations of the first few extremal Hilbert modular forms are given in the table in 
Example \ref{Table} below.

\subsection{Golden lattices}

\begin{definition} 
An even unimodular $R$-lattice $(\Lambda , Q)$ is called a {\em golden lattice},
if its Hilbert theta series is an extremal symmetric Hilbert modular form.
\end{definition}

%

\begin{prop}\label{goldenex}
Let $(L,q)$ be an even unimodular lattice of dimension $N$ and 
$\vartheta \in \End _{\Z} (L)$ be a symmetric endomorphism of $L$ with minimal polynomial 
$X^2+X-1$. 
Then $L$ is a $\Z[\vartheta ]$-lattice $\Lambda $ and 
$L = L_{\eta ^{-1}}$ for the even unimodular $\Z[\vartheta ]$-lattice $(\Lambda , Q) $ with quadratic form 
$$Q:\Lambda \to R,  Q(\lambda ):=  \frac{1}{2}(q(\lambda ) +  q(\vartheta \lambda )) + 
\frac{1}{2}(q(\vartheta \lambda )-q(\lambda ) ) \sqrt{5}  .$$
Assume that there is some automorphism $\sigma \in \Aut(L,q)$ such that 
$\sigma \vartheta = (-1-\vartheta) \sigma $. Then 
for $N=8,16,24,32,48,56,72$ the lattice $(\Lambda , Q)$ is a golden lattice,
if and only if $(L,q)$ is an extremal even unimodular lattice.
\end{prop}

\begin{proof}
Clearly any such endomorphism $\vartheta $ defines a $\Z[\vartheta ]$-structure on the $\Z$-lattice $L$.
Since $\vartheta $ is a symmetric endomorphism, the form $q$ is a trace form, 
$q(\lambda ) =  \Trace_{K/\Q} (\eta^{-1} Q( \lambda ))$ for some $K$-valued quadratic form $Q$. 
If $Q(\lambda ) = a + b\sqrt{5}$, then 
$$q(\lambda ) = \Trace_{K/\Q}  (\eta^{-1}  Q(\lambda ))  = a-b \mbox{ and }
q(\vartheta \lambda ) = \Trace _{K/Q} (\vartheta ^2 \eta^{-1} Q(\lambda )) = a+b$$ 
and hence $Q(\lambda ) = \frac{1}{2} (q(\lambda ) + q(\vartheta \lambda )) + \frac{1}{2} (q(\vartheta \lambda ) - q(\lambda )) \sqrt{5} $.
Clearly if $q(\lambda ) \in \Z$ and $q(\vartheta \lambda ) \in \Z$  then also $Q(\lambda ) \in R$,
therefore $(\Lambda, Q)$ is even. By Equation \eqref{Zdual} the lattice $(\Lambda, Q)$ is unimodular. 
\\
The Galois invariance of $\Theta (\Lambda , Q)$ follows from the fact that
 $Q(\vartheta \sigma (\lambda )) =  \overline{Q(\lambda )} $  for all $\lambda \in \Lambda $:
\\
To see this write $Q(\lambda ) = a+b\sqrt{5}$ and $Q(\sigma (\lambda )) = a'+b'\sqrt{5}$, so 
$$Q(\vartheta \sigma (\lambda )) = \vartheta ^2 Q(\sigma (\lambda )) = \frac{3a'-5b'}{2} + \frac{3b'-a'}{2} \sqrt{5} .$$
Then 
$$ \begin{array}{lcl} 
\Trace _{K/\Q} (\eta ^{-1} Q(\lambda )) = \Trace _{K/\Q} (\eta ^{-1} Q(\sigma (\lambda ))) & \mbox{ yields } & a-b=a'-b' \\ 
\Trace _{K/\Q} (\eta ^{-1} Q(\vartheta \lambda )) = \Trace _{K/\Q} (\eta ^{-1} Q(\sigma (\vartheta \lambda )))  \Leftrightarrow &  & \\ 
\Trace _{K/\Q} (\eta ^{-1}\vartheta ^2 Q(\lambda )) = 
\Trace _{K/\Q} (\eta ^{-1} \overline{\vartheta }^2  Q(\sigma ( \lambda ))) & \mbox{ yields } &  2a-4b = a'+b'  
\end{array} $$
so in total $a'=\frac{3a-5b}{2}, b' = \frac{a-3b}{2} $ which gives 
$$Q(\sigma (\lambda )) = \overline{\vartheta^2 Q(\lambda ) } \mbox{ for all } \lambda \in \Lambda .$$
Therefore the mapping $\lambda \mapsto \vartheta \sigma (\lambda ) $ gives a bijection between 
$$\{ \lambda \in \Lambda \mid Q(\lambda ) = \alpha \} \mbox{ and } 
\{ \lambda \in \Lambda \mid Q(\lambda ) = \overline{\alpha } \}  $$
so the Hilbert theta series of $(\Lambda , Q)$ is symmetric.
\\
The last statement follows from explicit computations in the ring of Hilbert modular forms.
These show that for weight 
$2,4,6,8,12,14,$ and $18$ the condition that $f(q_0,1)$ be an extremal elliptic modular form 
and that $f(q_0,q_1)$ has non negative coefficients imply that $f$ is an extremal Hilbert modular form.
The opposite direction follows from the table in Example \ref{Table}.
\end{proof}


\subsection{Associated modular lattices} 

Generalising unimodular lattices
Quebbemann \cite{Quebbemann} introduced the notion of $p$-modular lattices.

\begin{definition}
An even $\Z$-lattice $L$ in euclidean space is called {\em $p$-modular}, if there is a 
similarity $\sigma $ of norm $N$ (so $(\sigma (v ),\sigma (w)) = p (v,w) $ for all 
$v,w\in L $) such that $ \sigma (L^*) = L$. 
\end{definition}

If $p$ is one of the 6  primes for which $p+1$ divides 24, then the theta series of
$p$-modular lattices generate a polynomial ring with 2 generators from which one
obtains a similar notion of extremality as for unimodular lattices:
Let $k:=\frac{24}{p+1}$. Then any $p$-modular lattice $L$ of dimension $N$ 
satisfies $\min (L) \leq 1 + \lfloor \frac{N}{2k} \rfloor $. 
$p$-modular lattices achieving this bound are called {\em extremal}.

\begin{example}\label{Table}
For the first few weights the extremal Hilbert modular forms $f$ turn out to be unique
and start with non-negative integral coefficients. 
The valuation  $\nu (f-1) = (s,t) $ can be read off from the following table.
Any golden lattice $\Lambda $ with theta series $f$ defines an even unimodular
 $\Z $-lattice $L_{\eta^{-1} }$ of minimum $s$  and a $5$-modular lattice
$L_1$ of minimum $t$.
The Hilbert modular form $f$ also gives us information about the minimal vectors
$\Min (L,q):= \{ v\in L \mid q(v) = \min (L) \} $. 
The kissing number $s_{\eta ^{-1}} = |\Min (L_{\eta ^{-1}})| $ and $s_1 = | \Min (L_1)| $ of these two lattices
can be read off from column 3 and 4 of the table.
That all minimal vectors of $L_1$ are also minimal vectors of $L_{\eta ^{-1}}$ is indicated by
a $+$ in the last column. A $-$ means that only half of the minimal vectors of $L_1$ are
also contained in $\Min (L_{\eta ^{-1}})$.
$$
\begin{array}{|r|c|r|r|c|}
\hline
\mbox{weight}  & 
\nu (f-1) & s_{\eta^{-1} }  & s_1  & \subset \\ \hline
2 & (1,2) & 240 & 120 & + \\
4 & (1,2) & 480 & 240 & + \\
6 & (2,4) & 196560 & 37800 & + \\
8 & (2,4) & 146880  & 21600 & + \\
10 & (2,5) & 39600 & 79200 & - \\
12 & (3,6) & 52416000  & 2620800 & + \\
14 & (3,6) & 15590400 & 537600 & + \\
16  & (3,7) & 2611200 & 2611200 & - \\
18  & (4,8) &  6218175600  & 75411000 & + \\
20 & (4,9)  & 1250172000  & 609840000 & - \\
24 & (5,10)  & 565866362880 & 1655821440 & +  \\
30  & (6,13)  &45792819072000 & 3217294080000  & - \\
\hline
\end{array}
$$
\end{example}

Note that 
\cite[Proposition 3.3]{seoul} applied to  an even unimodular
$R$-lattice $(\Lambda ,Q) $ gives that 
$$\frac{5}{2} \min (L_{\eta ^{-1}} ) \geq \min (L_1) \geq 2  \min (L_{\eta^{-1}}) .$$
The argument is that $\eta = 1 + \overline{\vartheta }^2$ and 
$5\eta ^{-1} = 1+\vartheta ^2$ so for all $\lambda \in \Lambda $
\begin{equation}\label{sumsum} 
\begin{array}{ll} 
Q(\lambda ) = (1+\overline{\vartheta }^2) \eta ^{-1} Q(\lambda ) = \eta ^{-1} Q(\lambda ) + \eta ^{-1} Q(
\overline{\vartheta } \lambda ) & \mbox{ and } \\
5 \eta ^{-1} Q(\lambda ) = (1+\vartheta ^2) Q(\lambda ) = Q(\lambda ) + Q(\vartheta \lambda ) 
\end{array} 
\end{equation}
Immitating the proof of this proposition we find the following bound.

\begin{prop}\label{nu}
Let $f$ be a symmetric Hilbert modular form of weight $w$, 
 $s:= 1+\lfloor \frac{w}{6} \rfloor $ and $t:= \lfloor \frac{5s}{2} \rfloor $.
Then $\nu (f-1) \leq (s,t)$. More precisely 
put $\nu (f-1) = (s',t')$. Then $s'\leq s$ and
$ \lfloor \frac{5s'}{2} \rfloor \geq t' \geq 2s' .$
\end{prop}

\begin{proof}
Let $\nu (f-1) = (s',t')$. Since $f(q_0,1)$ is an elliptic modular 
form of weight $2w$ we get that $s' \leq s$. 
To obtain the other two inequalities write 
$$f= 1+\sum _{X\in R_{+}} A_X q_0 ^{\Trace(\eta^{-1} X)} q_1 ^{\Trace(X)} .$$ 
Then $f$ is invariant under the transformation by $\diag (u,u^{-1} ) \in \SL_2(R)$ 
for any $u\in R^*$. 
This shows that $$A_X = A_{\vartheta ^2X} = A _{\overline{\vartheta }^2 X } .$$
To see that $t' \geq 2s'$ choose some totally positive $X = a+b\sqrt{5} \in R $ such that 
$t'=2a$ and $A_X = A _{\overline{\vartheta }^2 X } \neq 0$.  Then 
$$ a-b = \Trace (\eta ^{-1} X) \geq s' \mbox{ and } a+b = \Trace (\eta^{-1} \overline{\vartheta }^2 X ) \geq s' $$
and so $t'=2a=(a-b) + (a+b) \geq 2s' $.
To see that $t' \leq \frac{5s'}{2} $ we let $X:=a+b\sqrt{5} \in R$ be a 
totally positive element with
$\Trace _{K/\Q} (\eta ^{-1} X) = a-b = s'$  and $A_X \neq 0$.
Then 
$$2a = \Trace _{K/\Q} ( X)  \geq t' \mbox{ and } 3a-5b = \Trace_{K/\Q} (\vartheta ^2 X ) \geq t' $$
and hence $5s' = 5 a-5b \geq 2 t'$.
\end{proof}

The minimum of an extremal 5-modular lattice 
is  $1+\lfloor \frac{N}{8} \rfloor$ an for large $N$ this is  strictly bigger than 
$\frac{5}{2} (1+\lfloor \frac{N}{24} \rfloor ) $ which yields:

\begin{cor}
Let $L$ be a golden lattice of dimension  $N$. Then $L_1$ is an extremal 5-modular lattice
if and only if $N=8$ or $N=24$.
\end{cor}


Any golden lattice defines a family of modular lattices:

\begin{theorem}\label{main}
Let $(\Lambda ,Q)$ a golden lattice of dimension $n$ and $(s,t):= \nu (\Theta (\Lambda ,Q) -1) $.
For $a\in \N_0$ the trace lattice $L_{1+a\eta^{-1} }$ is an $(a^2+5a+5)$-modular lattice of 
minimum $\geq t+as$. 
\end{theorem}

\begin{proof}
Recall that $L_{1+a\eta^{-1} } = (\Lambda , \Trace _{K/\Q} ((1+a\eta^{-1} ) Q ) ) $. 
Since $\Trace _{K/\Q }(\eta^{-1} Q)$ and $\Trace _{K/\Q } (Q)$ are positive 
definite and take integral values on $\Lambda $,
the lattice $L_{1+a\eta^{-1} }$ is even and positive definite for any $a\in \Z _{\geq 0}$. 
Clearly $\min (L_{1+a\eta^{-1} }) \geq \min (L_1) + a \min (L_{\eta^{-1} }) = t+as $. 
Now $(\Lambda , Q)$ is unimodular, so Equation \eqref{Zdual} yields that the $\Z $-dual of $L_{1+a\eta^{-1} }$ is
$$L_{1+a\eta^{-1} }^* = \eta^{-1} (1+a\eta^{-1} )^{-1}  \Lambda ^{\#} = 
\eta^{-1} (1+a\eta^{-1} )^{-1}  L_{1+a\eta^{-1}} .$$ 
The element $\eta  (1+a\eta^{-1} ) \in K $ hence defines a similarity between $L_{1+a\eta^{-1} }^*$ 
and $L_{1+a\eta^{-1} }$ or norm 
$N(\eta +a ) = a^2+5a+5 .$
\end{proof}

\section{Examples}

All even unimodular $\Z[\vartheta ]$-lattices are classified in dimension 4, 8, and 12 
\cite{Hsia}, \cite{Hsiaw5}. In each of these dimensions there is a unique golden lattice. 
For the other dimensions 16 to 36 we inspect automorphisms of some known extremal even unimodular lattice 
to find a golden lattice with the method from Proposition \ref{goldenex}. 
The tensor symbol $\otimes $ denotes the Kronecker product of matrix groups,
which is group theoretically the central product. 

\subsection{Dimension 4}\label{dim4} 
Already Maass \cite{Maass} has shown that there is a unique golden lattice, $F_4$,  of dimension 4.
It can be constructed from the maximal order $\mathcal M$ in the definite quaternion algebra 
with center $K$ that is only ramified at the two infinite places. 
Its $R$-automorphism group is 
$$\Aut _R(F_4) \cong (\SL_2(5) \otimes  \SL_2(5) ) : 2 .$$

\subsection{Dimension 8} 
Maass also showed that the only 8-dimensional even unimodular $R$-lattice is the 
orthogonal sum $F_4 \perp F_4$. 

\subsection{Dimension 12} 
The golden lattices of dimension 12 are exactly the $\Z[\vartheta ]$-structures of the 
unique extremal even unimodular $\Z $-lattice of dimension 24, the Leech lattice.
\cite{Hsiaw5} shows that there is a unique such golden lattice $\Lambda $.
Its automorphism group is 
$$\Aut _R (\Lambda ) \cong 2.J_2  \otimes \SL_2(5) .$$

\subsection{Dimension 16} 

By \cite[Table (1.2)]{Hsia} the mass of all even unimodular $R$-lattices of rank 16 is $>10^6$,
so a complete classification seems to be out of reach. 
Here it would be desirable to have a mass formula for the lattices without roots
in analogy to the classical case of even unimodular $\Z$-lattices \cite{King}. 

There are several known extremal even unimodular lattices in dimension 32 which have
a fairly big automorphism group.
In particular there are two golden lattices $\Lambda _1 $ and $\Lambda _2$ that have are $\mathcal M$-lattices
for $\mathcal M$ as in \ref{dim4}
(see \cite{coul1} and \cite[Table 2]{seoul} for the automorphism group).
The automorphism groups $G_i = \Aut _{\Z[\vartheta ]} (\Lambda _i ) $ are 
$$G_1 \cong (\otimes ^4 \SL_2(5) ) :S_4   ,\ G_2 \cong 
\SL_2(5) \otimes 2^{1+6}_-.O_6^-(2) .$$

\subsection{Dimension 20} 

No golden lattice of dimension 20 is known. 
It is an interesting problem to construct such a golden lattice or to 
prove its non-existence, since this is the smallest dimension for which 
$\min (L_1) > 2 \min (L_{\eta ^{-1}})$.

From the extremal even unimodular lattice $L$ in \cite{DB} with automorphism group 
$(U_5(2) \times 2^{1+4}_-.Alt_5).2$ and an automorphism $z\in L$ of order 5 with irreducible 
minimal polynomial one obtains a Galois invariant $\Z[z+z^{-1} ] $-lattice $(\Lambda ,Q)$ 
for which the 5-modular trace lattice $L_1$ has $19800$ minimal vectors of norm 4 and
no vectors of norm 5. 

\subsection{Dimension 24} 

The extremal even unimodular lattice $P_{48n}$ constructed in \cite[Theorem 5.3]{cyclo} 
has an obvious structure as a $\Z[\vartheta ]$-lattice with automorphism group 
$\SL_2(13).2 \otimes \SL_2(5) $. This provides one example of a golden lattice of 
$\Z [\vartheta ]$-dimension 24.


\subsection{Dimension 32} 

No golden lattice of dimension 32 is known.
There is one extremal even unimodular lattice $L$ of dimension 64 constructed in \cite[p. 496]{cyclo}
of which the extremality is proven in \cite[Proposition 4.2]{seoul}.
The subgroup $G:=\SL_2(17) \otimes \SL_2(5) $ of the automorphism group of  $L$ 
has endomorphism ring $\Z [(1+\sqrt{17})/2,\vartheta ]$ and with Proposition \ref{goldenex} 
the structure over $\Z [\vartheta ]$ yields a  lattice $(\Lambda , Q)$
whose Hilbert theta series is symmetric, $L=L_{\eta ^{-1}}$ is extremal, but the 
minimum of the 5-modular lattice $L_1$ is only 6 (and not 7, as required for a golden lattice). 

\subsection{Dimension 36} 

In \cite{dim72}  an extremal even unimodular lattice $\Gamma _{72}$ of dimension 72 is constructed.
Any Galois invariant $\Z[\vartheta ]$-structure on $\Gamma _{72}$ will give rise to a golden lattice of rank 36. 
The lattice $\Gamma _{72}$ can be obtained has a Hermitian tensor product \cite{coul} 
$$ \Gamma _{72}= P_b \otimes _{\Z[\frac{1+\sqrt{-7}}{2}] } P $$ 
where $P$ is the $\Z[\frac{1+\sqrt{-7}}{2}] $-structure of the Leech lattice with 
automorphism group $\SL_2(25) $. 
The group $\SL_2(25) \leq \GL_{24} (\Z )$ contains an element $\zeta $ of order 
5 with irreducible minimal polynomial. Put $\vartheta := \zeta + \zeta ^{-1} $. 
Then $P$ is a $\Z[\vartheta , \frac{1+\sqrt{-7}}{2} ]$-lattice with automorphism group 
$U:=(C_5\times C_5 ) : C_4$. This yields a $\Z[\vartheta ]$-structure on the Hermitian 
tensor product $\Gamma_{72} $ which defines a golden lattice of rank 36 whose automorphism 
group contains $U\times PSL_2(7)$.



\begin{thebibliography}{10}
\bibitem{MAGMA} 
W. Bosma, J. Cannon,  C. Playoust, The Magma algebra system. I. The user language. J. Symbolic Comput., 24(3-4):235-265, 1997
\bibitem{DB} 
G. Nebe, N.J.A. Sloane, A catalogue of lattices. 
http://www.math.rwth-aachen.de/$\sim $Gabriele.Nebe/LATTICES/
\bibitem{coul1} 
R. Coulangeon,  R\'eseaux unimodulaires quaternioniens en dimension $\leq 32$. Acta Arith. 70 (1995) 9-24.
\bibitem{coul} 
R. Coulangeon, G. Nebe, The unreasonable effectiveness of tensor products. this issue.
\bibitem{Ebeling}
W. Ebeling, {\em Lattices and Codes.} Vieweg 1994.
\bibitem{Gundlach}
K.-B. Gundlach, 
 Die Bestimmung der Funktionen zur Hilbertschen Modulgruppe des Zahlk\"orpers $Q(\sqrt{5})$.  Math. Ann. 152 
(1963) 226-256.
\bibitem{Hsia}
J.S. Hsia,
Even positive definite unimodular quadratic forms over real quadratic fields.
Rocky Mountain J. Math. 19 (1989) 725-733. 
\bibitem{Hsiaw5} P. J. Costello, J.S. Hsia,
 Even unimodular 12-dimensional quadratic forms over 
${\bf Q}(\sqrt5)$.
 Adv. in Math.  64  (1987) 241-278.
\bibitem{Hsiaw2} J.S.  Hsia, D.C.  Hung, 
 Even unimodular 8-dimensional quadratic forms over ${\bf Q}(\sqrt2)$.
 Math. Ann.  283  (1989) 367-374.
\bibitem{Hungw3}
D.C. Hung,  Even positive definite unimodular quadratic forms over 
${\bf Q}(\sqrt3)$. Math. Comp. 57 (1991) 351-368. 
\bibitem{King}
O. King,
A mass formula for unimodular lattices with no roots.
 Math. Comp. 72 (2003) 839-863
\bibitem{Maass} H. Maass, Modulformen und quadratische Formen \"uber dem quadratischen Zahlk\"orper $R(\sqrt{5})$.
Math. Ann. 118 (1941) 65-84.
\bibitem{cyclo} G. Nebe, 
Some cyclo quaternionic lattices.
J. Algebra 199, 472-498 (1998) 
\bibitem{seoul} G. Nebe, 
Construction and investigation of lattices with matrix groups.
in Myung-Hwan Kim, John S. Hsia, Y. Kitaoka, R. Schulze-Pillot (Ed.) 
{\em Integral Quadratic Forms and Lattices}, Contemporary Mathematics 249 (1999)  205-220.
\bibitem{dim72} G. Nebe, 
An even unimodular 72-dimensional lattice of minimum 8. J. Reine und Angew. Math. (to appear) 
\bibitem{Quebbemann}
H.-G. Quebbemann,  Modular lattices in Euclidean spaces.
 J. Number Theory 54 (1995) 190-202.
\end{thebibliography}
\end{document}